\documentclass{amsart}
\usepackage{mathrsfs,bbm}

\theoremstyle{plain}
\newtheorem{thm}{Theorem}[section]
\newtheorem{lemma}[thm]{Lemma}
\newtheorem{cor}[thm]{Corollary}
\newtheorem{prop}[thm]{Proposition}

\theoremstyle{definition}

\newtheorem{remark}[thm]{Remark}

\newtheorem{example}[thm]{Example}
\newtheorem{problem}[thm]{Problem}

\theoremstyle{remark}
\newtheorem*{condition1}{\bf Condition~(C1)}
\newtheorem*{condition2}{\bf Condition~(C2)}

\newcommand\mstrut{{\phantom{.}}}
\newcommand\bull{{\hbox{\bf .}}}
\newcommand\newdot{{\kern.8pt\cdot\kern.8pt}}

\font\sevenrm=cmr7

\newcommand\E{\mathbb{E}}
\renewcommand{\H}{{\mathbb H}}

\newcommand\R{\mathbb{R}}

\newcommand\1{\hbox{\kern.375em\vrule height1.57ex depth-.1ex width.05em\kern-.375em 1}}

\newcommand\SA{\mathscr A}

\newcommand\SF{\mathscr F}

\catcode`\@=11
\def\acong{\mathrel{\mathpalette\@avereq\sim}} 
\def\@avereq#1#2{\lower.5\p@\vbox{\baselineskip\z@skip\lineskip-.5\p@
  \ialign{$\m@th#1\hfil##\hfil$\crcr#2\crcr\longrightarrow\crcr}}}
\def\mequal{\mathrel{\mathpalette\@mvereq{\hbox{\sevenrm m}}}} 
\def\@mvereq#1#2{\lower.5\p@\vbox{\baselineskip\z@skip\lineskip1.5\p@
    \ialign{$\m@th#1\hfil##\hfil$\crcr#2\crcr=\crcr}}}
\def\tr#1#2{/\!/_{\!#1,#2}^{\phantom{.}}} 
\def\itr#1#2{/\!/_{\!#1,#2}^{-1}}  
\def\map#1#2#3{{#1}\colon\,{#2}\to{#3}}
\def\nbull{{\raise1.5pt\hbox{\bf .}}}
\def\tbull{{\raise1.0pt\hbox{\bf .}}}
\def\nab#1{\nabla_{\!#1}^\mstrut}
\def\lambdamin{\lambda_{\hbox{\sevenrm min}}}

\newbox\ovlbox
\def\ovl#1{\setbox\ovlbox\hbox{$#1$}\rlap{\kern.5\wd\ovlbox\kern-1.5pt
  $\overline{\hbox to4pt{\hss$\phantom{#1}$\hss}}$\hss}#1}
\def\unl#1{\setbox\ovlbox\hbox{$#1$}\rlap{\kern.5\wd\ovlbox\kern-2.5pt
  $\underline{\hbox to4pt{\hss$\phantom{#1}$\hss}}$\hss}#1}

\def\mathpal#1{\mathop{\mathchoice{\text{\rm #1}}%
   {\text{\rm #1}}{\text{\rm #1}}%
   {\text{\rm #1}}}\nolimits}
\def\id{\mathpal{id}}
\def\vol{\mathpal{vol}}

\begin{document}

\title [The differentiation of hypoelliptic diffusion semigroups] {The
  differentiation of hypoelliptic diffusion semigroups}

\author[M. Arnaudon]{Marc Arnaudon} \address{ Laboratoire de Math\'ematiques
  et Applications, CNRS: UMR6086\hfill\break\indent Universit\'e de Poitiers,
  T\'el\'eport 2 -- BP 30179\hfill\break\indent F--86962 Futuroscope
  Chasseneuil, France} \email{marc.arnaudon@math.univ-poitiers.fr}

\author[A. Thalmaier]{Anton Thalmaier} \address{ Unit\'e de Recherche en
  Math\'ematiques, FSTC\hfill\break\indent Universit\'e du
  Luxembourg\hfill\break\indent 6, rue Richard
  Coudenhove-Kalergi\hfill\break\indent L--1359 Luxembourg, Grand-Duchy of
  Luxembourg} \email{anton.thalmaier@uni.lu} \date{\today\ \emph{ File:
  }\jobname.tex}

%
%
\begin{abstract}\noindent
  Basic derivative formulas are presented for hypoelliptic heat semigroups and
  harmonic functions extending earlier work in the elliptic case.  Following
  the approach of~\cite{TH 97}, emphasis is placed on developing integration
  by parts formulas at the level of local martingales.  Combined with the
  optional sampling theorem, this turns out to be an efficient way of dealing
  with boundary conditions, as well as with finite lifetime of the underlying
  diffusion.  Our formulas require hypoellipticity of the diffusion in the
  sense of Malliavin calculus (integrability of the inverse Malliavin
  covariance) and are formulated in terms of the derivative flow, the
  Malliavin covariance and its inverse.
  Finally some extensions to the nonlinear setting of harmonic mappings are
  discussed.
\end{abstract}

\maketitle
\tableofcontents

\noindent\keywords{{\em Keywords}: Diffusion semigroup, hypoelliptic operator,
  integration by parts,\hfill\break Malliavin calculus, Malliavin covariance}

\smallskip\noindent 
%
%
\section{Introduction}\label{Sect1}\noindent
Let $M$ be a smooth $n$-dimensional manifold.  On $M$ consider a globally
defined Stratonovich SDE of the type
\begin{equation}
  \label{eq:StratonovichSDE}
  \delta X=A(X)\,\delta Z+A_0(X)\,dt
\end{equation}
with $A_0\in\Gamma(TM)$, $A\in\Gamma(\R^r\otimes TM)$ for some $r$, and $Z$ an
$\R^r$-valued Brownian motion on some filtered probability space satisfying
the usual completeness conditions.  Here $\Gamma(TM)$, resp.\
$\Gamma(\R^r\otimes TM)$, denote the smooth sections over $M$ of the tangent
bundle $TM$, resp.\ the vector bundle $\R^r\otimes TM$.

Solutions to \eqref{eq:StratonovichSDE} are diffusions with generator given in
H\"ormander form as
\begin{equation}
  \label{eq:Generator}
  L=A_0+{1\over 2}\sum\limits_{i=1}^r A_i^2
\end{equation}
where $A_i=A(\,\nbull\,)e_i\in\Gamma(TM)$ and $e_i$ the $i$th standard unit
vector in~$\R^r$.

There is a partial flow $X_t(\,\nbull\,)$, $\zeta(\,\nbull\,)$ to
\eqref{eq:StratonovichSDE} such that for each $x\in M$ the process $X_t(x)$,
$0\le t<\zeta(x)$, is the maximal strong solution to
\eqref{eq:StratonovichSDE} with starting point $X_0(x)=x$ and explosion
time~$\zeta(x)$. Adopting the notation $X_t(x,\omega)=X_t(x)(\omega)$, resp.\
$\zeta(x,\omega)=\zeta(x)(\omega)$ and
\begin{equation*}
  M_t(\omega)=\{x\in M\colon\ t<\zeta(x,\omega)\},
\end{equation*} 
it further means that there exists a set $\Omega_0\subset\Omega$ of full
measure such that for all $\omega\in\Omega_0$ the following conditions hold:

\begin{itemize}
\item[(i)] $M_t(\omega)$ is open in $M$ for $t\geq0$, i.e.\
  $\zeta(\,\nbull\,,\omega)$ is lower semicontinuous on $M$.
\item[(ii)] $\map{X_t(\,\nbull\,,\omega)}{M_t(\omega)}M$ is a diffeomorphism
  onto an open subset $R_t(\omega)$ of $M$.
\item[(iii)] For $t>0$ the map $s\mapsto X_s(\,\nbull\,,\omega)$ is continuous
  from $[0,t]$ to $C^\infty\bigl(M_t(\omega),M\bigr)$ when the latter is
  equipped with the $C^\infty$-topology.
\end{itemize}

\noindent Thus, the differential $\map{T_xX_t}{T_xM}{T_{X_t}M}$ of the map
$\map{X_t}{M_t}{M}$ is well-defined at each point $x\in M_t$, for all
$\omega\in\Omega_0$.  We also write $X_{t\ast}$ for $TX_t$.

Let
\begin{equation}
  \label{eq:SemiGroupMin}
  (P_tf)(x)=\E\bigl[\bigl(f\circ X_t(x)\bigr)\,1_{\{t<\zeta(x)\}}\bigr]
\end{equation}
be the minimal semigroup associated to \eqref{eq:StratonovichSDE}, acting on
bounded measurable functions $\map fM\R$.

Let $\hbox{Lie}\bigl(A_0,A_1,\dots,A_r\bigr)$ denote the Lie algebra generated
by $A_0,\dots,A_r$, i.e., the smallest $\R$-vector space of vector fields on
$M$ containing $A_0,\dots,A_r$ and being closed under Lie brackets.  We
suppose that \eqref{eq:Generator} is non-degenerate in the sense that the
ideal generated by $(A_1,\dots,A_r)$ in
$\hbox{Lie}\bigl(A_0,A_1,\dots,A_r\bigr)$ is the full tangent space at each
point $x\in M$:
$$
\hbox{Lie}\bigl(A_i,\,[A_0,A_i]\colon\,i=1,\dots,r\bigr)(x)=T_xM
\quad\hbox{for all }x\in M.\leqno\hbox{(H1)}
$$
Note that (H1) is equivalent to the following H\"ormander condition for
${\partial\over\partial t}+L$ on $\R\times M$:
\begin{equation*}
  \dim\hbox{Lie}\Bigl(\textstyle{\partial\over\partial t}
  +A_0, A_1,\dots,A_r\Bigr)(t,x)=n+1\quad\hbox{for all }(t,x)\in\R\times M.
\end{equation*}  
By H\"ormander's theorem, under (H1) the semigroup \eqref{eq:SemiGroupMin} is
strongly Feller (mapping bounded measurable functions on $M$ to bounded
continuous functions on $M$) and has a smooth density $p\in
C^\infty({]0,\infty[}\times M\times M)$ such that
\begin{equation*}
  P\bigl\{X_t(x)\in dy,\ t<\zeta(x)\bigr\}=p(t,x,y)\vol(dy),\quad t>0,\ x\in M,
\end{equation*}
see \cite{BI 81} for a probabilistic discussion.

In this paper we are concerned with the problem of finding stochastic
representations, under hypothesis (H1), for the derivative $d(P_tf)$ of
\eqref{eq:SemiGroupMin} which do not involve derivatives of $f$.  Analogously,
in the situation of $L$-harmonic functions $\map uD\R$, given on some domain
$D$ in $M$ by its boundary values $u\vert\partial D$ via
\begin{equation}
  \label{eq:ReprHarmFunct}
  u(x)=\E\,[u\circ X_{\tau(x)}(x)], 
\end{equation}
formulas are developed for $du$ not involving derivatives of the boundary
function; here $\tau(x)$ is the first exit time of $X(x)$ from $D$.

The paper is organized as follows.  In Section~\ref{Sect2} we collect some background on
Malliavin calculus related to hypoelliptic diffusions.  In Section~\ref{Sect3} we
explain our approach to integration by parts in the hypoelliptic case which
leads to differentiation formulas for hypoelliptic semigroups.  Section~\ref{Sect4} is
devoted to integration by parts formulas at the level of local martingales.
In Section~\ref{Sect5} control theoretic aspects related to differentiation formulas are
discussed.  It is shown that the solvability of a certain control problem
leads to simple formulas in particular cases, however the method turns out not
to cover the full hypoelliptic situation.  We deal with the general situation
in Section~\ref{Sect7} where we refine the arguments of Section~\ref{Sect4} and \ref{Sect5} to give
probabilistic representations for the derivative of semigroups and
$L$-harmonic functions in the hypoelliptic case.  A crucial step in this
approach is the use of the optional sampling theorem to obtain 
local formulas by appropriate stopping times, as in the elliptic case \cite{TH 97},
\cite{T-W 98}.  Our formulas are in terms of the derivative flow and
Malliavin's covariance; hence they are neither unique nor intrinsic: the 
appearing terms depend on the specific SDE and not just on the generator.

Finally, in Section~\ref{Sect8}, we deal with possible extensions to nonlinear
situations, like the case of harmonic maps and nonlinear heat equations for
maps taking values in curved targets.

All presented formulas do not require full H\"ormander's Lie algebra
condition~(H1) but rather invertibility and integrability of the inverse
Malliavin covariance which is known to be slightly weaker, but still
sufficient to imply hypoellipticity of $\frac\partial{\partial t}+L$.  In
particular, (H1) is allowed to fail on a collection of hypersurfaces.  The
reader is referred to \cite{B-M 95} for precise statements in this direction.

\section{Hypoellipticity and the Malliavin Covariance}\label{Sect2}
\setcounter{equation}0\noindent Let $B\in\Gamma(TM)$ be a vector field on $M$.
We consider the push-forward $X_{t\ast}B$ (resp.\ pull-back $X_{t\ast}^{-1}B$)
of $B$ under the partial flow $X_t(\,\nbull\,)$ to the
system~\eqref{eq:StratonovichSDE}, more precisely,
\begin{equation}
  \label{eq:PushedVF}
  \begin{split}
    (X_{t\ast}^\mstrut B)_x=\bigl(T_{X_t^{-1}(x)}X_t\bigr)\,B_{X_t^{-1}(x)}\,
    &,\quad x\in R_t,\\
    (X_{t\ast}^{-1}B)_x=\bigl(T_{X_t(x)}^\mstrut X_t\bigr)^{-1}\,B_{
      X_t(x)}^\mstrut\, &,\quad x\in M_t.
  \end{split}
\end{equation}
Note that $X_{t\ast}B$, resp.\ $X_{t\ast}^{-1}B$, are smooth vector fields on
$R_t$, resp.\ $M_t$, well-defined for all $\omega\in\Omega_0$.  By definition,
\begin{equation}
  \label{eq:PushedVFRewritten}
  \begin{split}
    (X_{t\ast}^\mstrut B)_x\,f&=B_{X_t^{-1}(x)}\,(f\circ X_t)\,,
    \quad x\in R_t,\\
    (X_{t\ast}^{-1}B)_x\,f&=B_{X_t(x)}^\mstrut\,(f\circ X_t^{-1})\,, \quad
    x\in M_t,
  \end{split}
\end{equation}
for germs $f$ of smooth functions at $x$.

\begin{thm}
  \label{SDETransp}
  The pushed vector fields $X_{t\ast}^\mstrut B$ and $X_{t\ast}^{-1}B$ as
  defined by {\rm\eqref{eq:PushedVF}} satisfy the following SDEs:
  \begin{align}
    \label{eq:SDEPushedVf}
    \delta(X_{t\ast}^\mstrut B) &=\sum_{i=1}^r\bigl[X_{t\ast}^\mstrut
    B,A_i\bigl]\,\delta Z^i_t
    +\bigl[X_{t\ast}^\mstrut B,A_0\bigl]\,dt\\
    \label{eq:SDEPushedVfInv}
    \delta(X_{t\ast}^{-1}B)
    &=\sum_{i=1}^r\bigl(X_{t\ast}^{-1}[A_i,B]\bigr)\,\delta Z^i_t
    +\bigl(X_{t\ast}^{-1}[A_0,B]\bigr)\,dt.
  \end{align}
\end{thm}

\begin{proof} 
  See Kunita {\cite{KU 81}}, section 5.
\end{proof}

We have the famous ``invertibility of the Malliavin covariance matrix'' under
the H\"ormander condition (H1), e.g., see~Bismut~\cite{BI 81}, Prop.~4.1.

\begin{thm}
  \label{MallInv}
  Suppose {\rm(H1)} holds.  Let $\sigma$ be a predictable stopping time, $x\in
  M$.  Then, a.s., for any predictable stopping time $\tau<\zeta(x)$, on
  $\{\sigma<\tau\}$
  \begin{equation*}
    \sum_{i=1}^r\int_\sigma^\tau
    (X_{s\ast}^{-1}A_i)_x^\mstrut\otimes(X_{s\ast}^{-1}A_i)_x^\mstrut\,ds
    \in T_xM\otimes T_xM
  \end{equation*} 
  is a positive definite quadratic form on $T_x^\ast M$.  In particular, a.s.,
  for each $t>0$,
  \begin{equation}
    \label{eq:MallCov}
    C_t(x)
    =\sum_{i=1}^r\int_0^t
    (X_{s\ast}^{-1}A_i)_x^\mstrut\otimes(X_{s\ast}^{-1}A_i)_x^\mstrut\,ds
  \end{equation}
  defines a positive symmetric bilinear form on $T_x^\ast M$ for $x\in M_t$.
\end{thm}

Thus, a.s., $C_t$ provides a smooth section of the bundle $TM\otimes TM$ over
$M_t$ with the property that all $C_t(x)$ are symmetric and positive definite.
We may choose a non-degenerate inner product $\langle\,\cdot,\cdot\,\rangle$
on $T_xM$ and read $C_t(x)\in T_xM\otimes T_xM$ as a positive definite
bilinear form on $T_xM$:
\begin{equation*}
  \bigl\langle C_t(x)u,v\bigr\rangle 
  =\sum_{i=1}^r\int_0^t\bigl\langle(X_{s\ast}^{-1}A_i)_x,u\bigr\rangle\,
  \bigl\langle(X_{s\ast}^{-1}A_i)_x,v\bigr\rangle\,ds,
  \quad u,v\in T_xM.
\end{equation*}
Under (H1) the ``random matrix'' $C_t(x)$ is invertible for $t>0$ and $x\in
M_t$.  The following property is a key point in the stochastic calculus of
variation, e.g., \cite{NO 86}, \cite{KU 90}, \cite{NU 95}.

\begin{remark}
  \label{AllLp}
  Under hypothesis (H1) and certain boundedness conditions on the vector
  fields $A_0,A_1,\dots,A_r$ (which are satisfied for instance if $M$ is
  compact) we have $(\det C_t(x))^{-1}\in L^p$ for all $1\leq p<\infty$.  In
  the same way,
  \begin{equation}
    \label{eq:InvMallCovLp}
    (\det C_\sigma(x))^{-1}\in L^p\quad\hbox{for $1\leq p<\infty$}
  \end{equation}
  if $\sigma=\tau_D^\mstrut(x)$ or $\sigma=\tau_D^\mstrut(x)\wedge t$ for some
  $t>0$ where $\tau_D^\mstrut(x)$ is the first exit time of $X_\tbull(x)$ from
  some relatively compact open neighbourhood $D\not=M$ of $x$.  Also note that
  $\tau_D^\mstrut(x)\in L^p$ for all $1\leq p<\infty$, e.g.~\cite{B-K-S 84},
  Lemma~(1.21).
\end{remark}

In the subsequent sections we adopt the following notation.  By definition,
$C_t(x)\in T_xM\otimes T_xM$, thus $\map{C_t(x)}{T^\ast_xM}{T_xM}$ and
$\map{C_t(x)^{-1}}{T_xM}{T^\ast_xM}$.  On the other hand,
\begin{equation}
  \label{eq:XinverseA}
  \map{(X_{s\ast}^{-1}A)_x^\mstrut}{\R^r}{T_xM},\quad
  z\mapsto\sum_{i=1}^r(X_{s\ast}^{-1}A_i)_x^\mstrut\,z^i.
\end{equation}
Let $\map{(X_{s\ast}^{-1}A)^\ast_x}{T^\ast_xM}{(\R^r)^\ast\equiv\R^r}$ be the
adjoint (dual) map to \eqref{eq:XinverseA}, then we may write
\begin{equation}
  \label{eq:MallCovRewritten}
  C_t(x)=\int_0^t (X_{s\ast}^{-1}A)_x^\mstrut\,(X_{s\ast}^{-1}A)^\ast_x\,ds
\end{equation}
for the Malliavin covariance.  In the sequel we usually identify $(\R^r)^\ast$
and $\R^r$.

\section{A Basic Integration by Parts Argument}\label{Sect3}
\setcounter{equation}0\noindent In this section we explain an elementary
strategy for integration by parts formulas which will serve us as a guideline
in the sequel.  The argument is inspired by Bismut's original approach to
Malliavin calculus~\cite{BI 81}.

Consider again the SDE \eqref{eq:StratonovichSDE} and assume (H1) to be
satisfied.  For simplicity, we suppose that $M$ is compact.  Let $a$ be a
predictable process taking values in $T_xM\otimes(\R^r)^\ast\equiv
T_xM\otimes\R^r$ and $\lambda\in T^\ast_xM$ such that for each $t>0$,
\begin{equation}
  \label{eq:CondGirsanov}
  \E\left[\exp\left({1\over2}\int_0^t\vert a_s\lambda\vert^2\,ds
    \right)\right]<\infty,\quad
  \hbox{ $\lambda$ locally about $0$.}
\end{equation}
Let $dZ^\lambda=dZ+a\lambda\,dt$ and consider the Girsanov exponential
$G^\lambda_\tbull$ defined by
\begin{equation}
  G^\lambda_t=\exp\left(-\int_0^t\,\langle a_s\lambda,dZ_s\rangle
    -{1\over2}\,\int_0^t\vert a_s\lambda\vert^2\,ds\right).
\end{equation}
Write $X^\lambda$ for the flow to our SDE driven by the perturbed BM
$Z^\lambda$, analogously $C_\tbull^\lambda(x)$ etc.  By definition,
$C_\tbull^\lambda(x)\in T_xM\otimes T_xM$ is a linear map from $T^\ast_xM$ to
$T_xM$ and $\map{C_\tbull^\lambda(x)^{-1}}{T_xM}{T^\ast_xM}$.

\begin{lemma}
  \label{DerMallInv}
  For any vector field $B\in\Gamma(TM)$ we have
  \begin{equation}
    \label{eq:FormulaLieBracket}
    {\partial\over\partial\lambda_k}
    \biggl\vert_{\lambda=0}(X_{t\ast}^\lambda)^{-1}(B)
    =\sum_{i=1}^r\left[\int_0^tX_{s\ast}^{-1}(A_i)\,a_s^{ik}\,ds\,,\,
      X_{t\ast}^{-1}(B)\right]
  \end{equation}
  in terms of the Lie bracket\/ $[\,,\,]$.
\end{lemma}
  
\begin{proof} 
  Note that $X_t^\lambda(x)=X_t\circ\varrho_t^\lambda(x)$ where
  $\varrho^\lambda(x)$ solves
  \begin{equation*}
    \left\lbrace\begin{aligned}
        d\varrho_t^\lambda&=X_{t\ast}^{-1}(A)(\varrho_t^\lambda)\,a_t\lambda\,dt\\ 
        \varrho_0^\lambda&=x.\end{aligned}\right.
  \end{equation*}
  In particular, we have
  \begin{equation*} {\partial\over\partial\lambda_k}
    \biggl\vert_{\lambda=0}\varrho_t^\lambda=
    \sum_{i=1}^r\int_0^t X_{s\ast}^{-1}(A_i)\,a_s^{ik}\,ds.
  \end{equation*}
  Moreover, from $X_{t\ast}^\lambda(x)
  =\bigl(T_{\varrho_t^\lambda(x)}X_t\bigr)(T_x\varrho_t^\lambda)$ we conclude
  that
$$
\bigl((X_{t\ast}^\lambda)^{-1}B\bigr){}_x=(T_x\varrho_t^\lambda)^{-1}
\bigl(T_{\varrho_t^\lambda(x)}X_t\bigr)^{-1}\,
B\bigl(X_t\circ\varrho_t^\lambda(x)\bigr)
\equiv(T_x\varrho_t^\lambda)^{-1}(X_{t\ast}^{-1}B)_{\varrho_t^\lambda(x)}.
$$
This gives the claim by definition of the bracket.
\end{proof}

\begin{thm}
  \label{ResultCrude}
  Let $M$ be compact and $f\in C^1(M)$. Assume that {\rm(H1)} is satisfied.
  Then, for each $v\in T_xM$,
  \begin{equation}
    \label{eq:FormulaCrude}
    d(P_tf)_xv=\E\left[\bigl(f\circ X_t(x)\bigr)\,\Phi_t\,v\right]  
  \end{equation}
  where $\Phi$ is an adapted process with values in $T_x^\ast M$ such that
  each $\Phi_t$ is $L^p$ for any $1\leq p<\infty$.
\end{thm}

\begin{proof} We fix $x$ and identify $T_xM$ with $\R^n$.  By Girsanov's
  theorem, for $v\in T_xM$, the expression
  \begin{equation*}
    H_k(\lambda)=\sum_\ell\E\left[\bigl(f\circ X_t^\lambda(x)\bigr)\cdot
      G_t^\lambda\cdot \bigl(C_t^\lambda(x)^{-1}\bigr)_{k\ell}\,v^\ell\right]
  \end{equation*}
  is independent of $\lambda$ for any $C^1$-function $f$ on $M$.  Thus
  \begin{equation*}
    \displaystyle\sum_k{\partial\over\partial\lambda_k}
    \biggl\vert_{\lambda=0}H_k(\lambda)=0
  \end{equation*}
  which gives
  \begin{align*}
    &\sum_{i,k,\ell}\E\left[\bigl(D_if\bigr)\bigl(X_t(x)\bigr)\biggl(X_{t\ast}
      \int_0^t(X_{s\ast}^{-1}A)_x^\mstrut\,a_s\,ds\biggr)_{ik}
      \bigl(C_t(x)^{-1}\bigr)_{k\ell}\,v^\ell\right]\\
    =-&\sum_{k,\ell}\E\left[f\bigl(X_t(x)\bigr)\,
      {\partial\over\partial\lambda_k}\biggl\vert_{\lambda=0}
      \bigl(G_t^\lambda\,\bigl(C_t^\lambda(x)^{-1}\bigr)_{k\ell}\,v^\ell\bigr)\right]\\
    =-&\sum_{k,\ell}\E\left[f\bigl(X_t(x)\bigr)
      \left(\left({\partial\over\partial\lambda_k}\biggl\vert_{\lambda=0}\!
          G_t^\lambda\right)\bigl(C_t(x)^{-1}\bigr)_{k\ell}
        +{\partial\over\partial\lambda_k}\biggl\vert_{\lambda=0}
        \bigl(C_t^\lambda(x)^{-1}\bigr)_{k\ell}\right)\,v^\ell \right].
  \end{align*}
  Note that
  \begin{equation*} {\partial\over\partial\lambda_k}\biggl\vert_{\lambda=0}
    G_t^\lambda=-\left(\int_0^t a_s^\ast\,dZ_s\right)_k
  \end{equation*}
  where $a^\ast$ taking values in $T_xM\otimes(\R^r)^\ast$ is defined as the
  adjoint to $a$.  Furthermore,
  \begin{equation*}
    {\partial\over\partial\lambda_k}\biggl\vert_{\lambda=0}C_t^\lambda(x)^{-1}
    =-C_t(x)^{-1}\,\left({\partial\over\partial\lambda_k}\biggl\vert_{\lambda=0}
      C_t^\lambda(x)\right)\,C_t(x)^{-1}.
  \end{equation*} 
  Recall that $(X_{s\ast}^{-1}A)_x\in(\R^r)^\ast\otimes T_xM$.  We set
  \begin{equation*}
    a_s=a^n_s=(X_{s\ast}^{-1}A)_x^\ast\,1_{\{s\leq\tau_n\}}\in T_xM\otimes(\R^r)^\ast
  \end{equation*} 
  where $(\tau_n)$ is an increasing sequence of stopping times such that
  $\tau_n\nearrow t$ and such that each $a^n_\bull$ satisfies condition
  \eqref{eq:CondGirsanov}.  This gives a formula of the type
  \begin{equation}
    \label{eq:ApproxFormula}
    \E\bigl[(df)_{X_t(x)}\,X_{t\ast}\,C_{\tau_n}(x)\,C_t(x)^{-1}v\bigr]
    =\E\left[\bigl(f\circ X_t(x)\bigr)\cdot\Phi_t^n\,v\right]
  \end{equation}
  Finally, taking the limit as $n\to\infty$, we get
  \begin{equation}
    \label{eq:FirstFormula}
    d(P_tf)_xv=\E\bigl[(df)_{X_t(x)}\,X_{t\ast}v\bigr]
    =\E\left[\bigl(f\circ X_t(x)\bigr)\cdot\Phi_t\,v\right]
  \end{equation}
  where
  \begin{align*}
    \Phi_t\,v
    =&\left(\int_0^t(X^{-1}_{s\ast}A)_x^\mstrut\,dZ_s\right)C_t^{-1}(x)\,v\\
    &+\sum_{k,\ell}\left(C_t(x)^{-1}\left({\partial\over\partial\lambda_k}
        \biggl\vert_{\lambda=0}
        C_t^\lambda(x)\right)\,C_t(x)^{-1}\right)_{k\ell}\,v^\ell
  \end{align*}
  which can be further evaluated by means of \eqref{eq:FormulaLieBracket}.
  Eq.\ \eqref{eq:FormulaLieBracket} also allows to conclude that
  $\Phi_t\in\cap_{p\geq1}L^p$.
\end{proof}

\section{Integration by Parts at the Level of Local Martingales}\label{Sect4}
\setcounter{equation}0\noindent Let $F(\,\nbull\,,X_\tbull(x))$, $x\in M$ be a
family of local martingales where $F$ is differentiable in the second variable
with a derivative jointly continuous in both variables.  We are mainly
interested in the following two cases:
\begin{align*}
  F(\,\nbull\,,X_\tbull(x))
  &=u\circ X_\tbull(x)\quad\hbox{for some $L$-harmonic function $u$ on $M$, and}\\
  F(\,\nbull\,,X_\tbull(x))&=(P_{t-\tbull}f)\bigl(X_\tbull(x)\bigr)
  \hskip.3cm\hbox{for some bounded measurable $f$ on $M$, $t>0$.}
\end{align*}
Let $dF$ denote the differential of $F$ with respect to the second variable.

\begin{thm}
  \label{LocMart}
  Let $F(\,\nbull\,,X_\tbull(x))$, $x\in M$ be a family of local martingales
  as described above.  Then, for any predictable $\R^r$-valued process $k$ in
  $L^2_{\hbox{\sevenrm loc}}(Z)$,
  \begin{equation}
    \label{eq:QuasiDer}
    dF(\,\nbull\,,X_\tbull(x))\,(T_xX_\tbull)
    \int_0^\bull (X_{s\ast}^{-1}A)_x^\mstrut k_s\,ds
    -F(\,\nbull\,,X_\tbull(x))\int_0^\bull\langle k,dZ\rangle,\quad x\in M,
  \end{equation}
  is a family of local martingales.
\end{thm}

\begin{proof}[Proof\/ {\rm(}by means of Girsanov\/{\rm)}]
  For $\varepsilon$ varying locally about $0$, consider the SDE
  \begin{equation}
    \label{eq:SDEPert}
    \delta X^\varepsilon=A(X^\varepsilon)\,\delta Z^\varepsilon
    +A_0(X^\varepsilon)\,dt  
  \end{equation}
  with the perturbed driving process
  $dZ^\varepsilon=dZ+\varepsilon\,k\,dt$. Then, for each $\varepsilon$,
  \begin{equation}
    \label{eq:PertMart}
    F\bigl(\,\nbull\,,X^\varepsilon_\tbull(x)\bigr)\,G^\varepsilon_\tbull
  \end{equation}
  is again a local martingale when the Girsanov exponential
  $G^\varepsilon_\tbull$ is defined by
  \begin{equation*}
    G^\varepsilon_r=\exp\Bigl(-\int_0^r\varepsilon\,\langle k,dZ\rangle
    -{1\over2}\,\varepsilon^2\!\int_0^r\vert k\vert^2\,ds\Bigr).
  \end{equation*}
  Moreover, the local martingale \eqref{eq:PertMart} depends $C^1$ on the
  parameter $\varepsilon$ (in the topology of compact convergence in
  probability), thus
  \begin{equation*}
    {\partial\over\partial\varepsilon}\Bigl\vert_{\varepsilon=0}
    F\bigl(\,\nbull\,,X^\varepsilon_\tbull(x)\bigr)\,G^\varepsilon_\tbull
    ={\partial\over\partial\varepsilon}\Bigl\vert_{\varepsilon=0}
    F\bigl(\,\nbull\,,X^\varepsilon_\tbull(x)\bigr)
    +F\bigl(\,\nbull\,,X_\tbull(x)\bigr)\,
    {\partial\over\partial\varepsilon}\Bigl\vert_{\varepsilon=0}G^\varepsilon_\tbull
  \end{equation*}
  is also a local martingale. Taking into account that
  \begin{equation*}
    {\partial\over\partial\varepsilon}\Bigl\vert_{\varepsilon=0}X^\varepsilon_r(x)
    =X_{r\ast}\int_0^r X_{s\ast}^{-1}A\bigl(X_s(x)\bigr)k_s\,ds
  \end{equation*}
  and
  \begin{equation*}
    {\partial\over\partial\varepsilon}\Bigl\vert_{\varepsilon=0}G^\varepsilon_r=
    -\int_0^r\langle k,dZ\rangle,
  \end{equation*}
  we get the claim.
\end{proof}

\begin{proof}[Alternative proof\/ {\rm(of Theorem \ref{LocMart}\rm)}]
  First note that $m_s:=dF(s,\,\nbull\,)_{X_s(x)}\,X_{s\ast}$, as the
  derivative of a family of local martingales, is a local martingale in
  $T^\ast_xM$, see~\cite{A-T 96}.  Thus also
  \begin{equation*}
    n_s:=m_sh_s-\int_0^sm_rdh_r
  \end{equation*}
  is a local martingale for any $T_xM$-valued adapted process $h$ locally of
  bounded variation. Choosing
  \begin{equation*}
    h=\int_0^\bull (X_{s\ast}^{-1}A)_x^\mstrut k_s\,ds
  \end{equation*}
  and taking into account that
  \begin{equation*}
    F(\,\nbull\,,X_\tbull(x))
    =\int_0^\bull dF(s,\,\nbull\,)_{X_s(x)}\,A\bigl(X_s(x)\bigr)\,dZ,
  \end{equation*}
  the claim follows by noting that
  \begin{equation*}
    \int_0^\bull dF(s,\,\nbull\,)_{X_s(x)}\,X_{s\ast}\,dh_s\mequal 
    F(\,\nbull\,,X_\tbull(x))\int_0^\bull\langle k,dZ\rangle
  \end{equation*}
  where $\mequal$ denotes equality modulo local martingales.
\end{proof}

Let $a$ be a predictable process taking values in $T_xM\otimes(\R^r)^\ast$ as
in the last section. The calculation above shows that
\begin{equation*}
  n_s:=dF(s,\nbull\,)_{X_s(x)}\,X_{s\ast}
  \left(\int_0^s (X_{r\ast}^{-1}A)_x^\mstrut\,a_r\,dr\right)
  -F\bigl(s,X_s(x)\bigr)\int_0^sa_r^\ast\,dZ_r
\end{equation*}
is a local martingale in $T_xM$ which implies that
\begin{equation*}
  N_s:=n_sh_s-\int_0^sn_r\,dh_r
\end{equation*}
is also a local martingale for any $T^\ast_xM$-valued adapted process $h$
locally of bounded variation.  In particular, choosing again
$a_s=(X_{s\ast}^{-1}A)_x^\ast$, we get
\begin{align*}
  N_s&=dF(s,\nbull\,)_{X_s(x)}\,X_{s\ast}\,C_s(x)\,h_s-
  F\bigl(s,X_s(x)\bigr)\left(\int_0^s(X_{r\ast}^{-1}A)_x\,dZ_r\right)h_s\\
  &-\int_0^s dF(r,\nbull\,)_{X_r(x)}\,X_{r\ast}\,C_r(x)\,dh_r +\int_0^s
  F\bigl(r,X_r(x)\bigr)
  \left(\int_0^r(X_{\rho\ast}^{-1}A)_x\,dZ_\rho\right)dh_r.
\end{align*}
For the last term it is trivial to observe that
\begin{align*}
  \int_0^s F\bigl(r,X_r(x)\bigr)
  &\left(\int_0^r(X_{\rho\ast}^{-1}A)_x\,dZ_\rho\right)dh_r\\
  &\hskip1cm\mequal F\bigl(s,X_s(x)\bigr)\int_0^s
  \left(\int_0^r(X_{\rho\ast}^{-1}A)_x\,dZ_\rho\right)dh_r.
\end{align*}
Now the idea is to take $h$ of the special form $h_s=C_s(x)^{-1}k_s$ for some
adapted $T_xM$-valued process $k$ locally pathwise of bounded variation such
that in addition $k_\tau=v$ and $k_s=0$ for $s$ close to~$0$.  Then the
remaining problem is to replace
\begin{equation}
  \label{eq:LousyTerm}
  \int_0^sdF(r,\nbull\,)_{X_r(x)}\,X_{r\ast}\,C_r(x)\,dh_r
\end{equation}
modulo local martingales by expressions not involving derivatives of $F$.
This however seems to be difficult in general, but in Section~\ref{Sect7} we show that,
more easily, the expectation of \eqref{eq:LousyTerm} can be rewritten in terms
not involving derivatives of $F$.

\section{Hypoelliptic Diffusions and Control Theory}\label{Sect5}
\setcounter{equation}0\noindent The following two corollaries are immediate
consequences of Theorem \ref{LocMart}.

\begin{cor}
  \label{LocMartPt}
  Let $\map fM\R$ be bounded measurable.  Fix $x\in M$ and $v\in T_xM$.  Then,
  for any predictable $\R^r$-valued process $k$ in $L^2_{\hbox{\sevenrm
      loc}}(Z)$,
  \begin{equation*}
    (dP_{t-\tbull}f)_{X_\bull(x)}\,(T_xX_\tbull)
    \Bigl[v+\!\int_0^\bull (X_{s\ast}^{-1}A)_x k_s\,ds\Bigr]
    -(P_{t-\tbull}f)\bigl(X_\tbull(x)\bigr)\!\int_0^\bull\langle k,dZ\rangle
  \end{equation*}
  is a local martingale on the interval $[0,t\wedge\zeta(x)[$.
\end{cor}

Note that $(dP_{t-\tbull}f)_{X_\bull(x)}\,(T_xX_\tbull)\,v$ is a local
martingale as the derivative of the local martingale
$(P_{t-\tbull}f)\bigl(X_\tbull(x)\bigr)$ at $x$ in the direction $v$,
see~\cite{A-T 96}.

\begin{cor}
  \label{LocMartHarmonic}
  Assume that $M$ is compact with nonempty smooth boundary $\partial M$. Let
  $u\in C(M)$ be $L$-harmonic on $M{\setminus}\partial M$.  Fix $x\in
  M{\setminus}\partial M$ and $v\in T_xM$.  Then, for any predictable
  $\R^r$-valued process $k$ in $L^2_{\hbox{\sevenrm loc}}(Z)$,
  \begin{equation*}
    (du)_{X_\bull(x)}\,(T_xX_\tbull)
    \Bigl[v+\!\int_0^\bull (X_{s\ast}^{-1}A)_xk_s\,ds\Bigr]
    -u\bigl(X_\tbull(x)\bigr)\!\int_0^\bull\langle k,dZ\rangle
  \end{equation*}
  is a local martingale on the interval $[0,\tau(x)[$ where $\tau(x)$ is the
  first hitting time of $X_\tbull(x)$ at $\partial M$.
\end{cor}

\begin{problem}[Control Problem]
  \label{MainProblem}
  Let $x\in M$ and $v\in T_xM$.  Consider the random dynamical system
  \begin{equation}
    \label{eq:ContrPr}
    \left\lbrace\begin{aligned}
        \dot h_s&=(X_{s\ast}^{-1}A)_x\,k_s\\  
        h_0&=v.
      \end{aligned}\right.
  \end{equation}
  Let $\sigma=\tau_D^\mstrut(x)$, resp., $\sigma=\tau_D^\mstrut(x)\wedge t$
  for some $t>0$, where $\tau_D^\mstrut(x)$ is the first exit time of
  $X_\tbull(x)$ from some relatively compact open neighbourhood $D$ of $x$.
  We are concerned with the problem of finding predictable processes $k$
  taking values in $\R^r$ such that $h_\sigma=0$,~a.s.
\end{problem}

\begin{example}
  \label{EllCase}
  Assume $L$ to be elliptic, i.e., $\map{A(x)}{\R^r}{T_xM}$ surjective for
  each $x\in M$. Then
  \begin{equation*}
    k_s=A^\ast\bigl(X_s(x)\bigr)\,T_xX_s\,\dot h_s
  \end{equation*}
  solves Problem \ref{MainProblem} if the terms are defined as follows:
  $A^\ast(\,\nbull\,)\in\Gamma(T^\ast M\otimes\R^r)$ is a smooth section and
  (pointwise) right-inverse to $A(\,\nbull\,)$,
  i.e.~$A(x)A^\ast(x)=\id_{T_xM}$ for $x\in M$, the process $h$ may be any
  adapted process with values in $T_xM$ and with absolutely continuous sample
  paths (e.g., paths in the Cameron-Martin space $\H(\R_+,T_xM)$) such that
  $h_0=v$ and $h_\sigma=0$,~a.s.  Thus, for elliptic $L$, there are
  ``controls'' $k$ transferring system \eqref {eq:ContrPr} from $v$ to $0$ in
  time $\sigma$, moreover it is even possible to follow prescribed
  trajectories $s\mapsto h_s$ from $v$ to $0$.  In the hypoelliptic case, this
  cannot be achieved in general, since the right-hand side in
  \begin{equation*}
    (T_xX_s)\,\dot h_s=A\bigl(X_s(x)\bigr)\,k_s
  \end{equation*} 
  is allowed to be degenerate.
\end{example}

Under the assumption that Problem \ref{MainProblem} has an affirmative
solution, we get differentiation formulas in a straightforward way.

\begin{thm}
  \label{ThmOne}
  Let $\map fM\R$ be bounded measurable, $x\in M$, $v\in T_xM$, $t>0$.  Let
  $D$ be a relatively compact open neighbourhood of $x$ and
  $\sigma=\tau_D^\mstrut(x)\wedge t$ where $\tau_D^\mstrut(x)$ is the first
  exit time of $X_\tbull(x)$ from $D$.  Suppose there exists an $\R^r$-valued
  predictable process $k$ such that
  \begin{equation*}
    \int_0^\sigma (X_{s\ast}^{-1}A)_x\,k_s\,ds\equiv v,\quad\hbox{a.s.,}
  \end{equation*}
  and $\bigl(\int_0^\sigma\vert k_s\vert^2\,ds\bigr)^{1/2}\in
  L^{1+\varepsilon}$ for some $\varepsilon>0$.  Then
  \begin{equation}
    \label{eq:dSemigroup}
    d(P_tf)_xv
    =\E\biggl[f\bigl(X_t(x)\bigr)\,1_{\{t<\zeta(x)\}}
    \int_0^\sigma\langle k,dZ\rangle\biggr]  
  \end{equation}
  where $P_tf$ is the minimal semigroup defined
  by~{\rm\eqref{eq:SemiGroupMin}}.
\end{thm}

\begin{proof}
  It is enough to check that the local martingale defined in Theorem
  \ref{LocMart} is actually a uniformly integrable martingale on the interval
  $[0,\sigma]$.  The claim then follows by taking expectations, noting that
  $(P_{t-\sigma}f)(X_\sigma(x))
  =\E^{\SF_\sigma}\bigl[f\bigl(X_t(x)\bigr)\,1_{\{t<\zeta(x)\}}\bigr]$.  See
  Theorem 2.4 in \cite{TH 97} for technical details.
\end{proof}

Along the same lines, now exploiting Corollary \ref{LocMartHarmonic}, the
following result can be derived.

\begin{thm}
  \label{ThmTwo}
  Let $M$ be compact with smooth boundary $\partial M\not=\emptyset$ and let
  $u\in C(M)$ be $L$-harmonic on $M{\setminus}\partial M$.  Let $x\in
  M{\setminus}\partial M$ and $v\in T_xM$.  Denote $\tau(x)$ the first hitting
  time of $X_\tbull(x)$ at $\partial M$.  Suppose there exists an
  $\R^r$-valued predictable process $k$ such that
  \begin{equation*}
    \int_0^{\tau(x)}(X_{s\ast}^{-1}A)_x^\mstrut\,k_s\,ds\equiv v,\quad\hbox{a.s.,}
  \end{equation*}
  and $\bigl(\int_0^{\tau(x)}\vert k_s\vert^2\,ds\bigr)^{1/2}\in
  L^{1+\varepsilon}$ for some $\varepsilon>0$.  Then the following formula
  holds:
  \begin{equation}
    \label{eq:dHarmonic}
    (du)_xv
    =\E\biggl[u\bigl(X_{\tau(x)}(x)\bigr)
    \int_0^{\tau(x)}\langle k,dZ\rangle\,\biggr].  
  \end{equation}
\end{thm}

\noindent
In the elliptic case, formulas of type \eqref{eq:dSemigroup} and
\eqref{eq:dHarmonic} have been used in~\cite{T-W 98} to establish gradient
estimates for $P_tf$ and for harmonic functions $u$, see also \cite{Driver-Thalm:2001} 
for extensions from functions to to sections.  
Nonlinear generalizations of the elliptic case, e.g., to harmonic maps and solutions of
the nonlinear heat equations, are treated in~\cite{A-T 97}.

As explained, differentiation formulas may be obtained from the local
martingales \eqref{eq:QuasiDer} by taking expectations if there is a
``control'' $(k_s)$ transferring the system \eqref{eq:ContrPr} from $h_0=v$ to
$h_\sigma=0$.  Solvability of the ``control problem'' is more or less
necessary for this approach, as is explained in the following remark.

\begin{remark}
  \label{QuasiDerivatives}
  Consider the general problem of finding semimartingales $h$, $\Phi$ with
  $h_0=v$ and $\Phi_0=0$ where $h$ is $T_xM$-valued and $\Phi$ real-valued
  such that
  \begin{equation}
    \label{eq:LocMartGeneral}
    n_s=(dF_s)_{X_s(x)}\,X_{s\ast}h_s+F_s(X_s(x))\,\Phi_s,\quad s\geq0  
  \end{equation}
  is a local martingale for any space-time transformation $F$ of the diffusion
  $X(x)$ such that $F_s(X_s(x))\equiv F(s,X_s(x))$ is a local martingale.  In
  the notion of quasiderivatives, as used by Krylov~\cite{KR 93, Krylov:2004}, this means
  that $\xi:=(T_xX)\,h$ is a $F$-quasiderivative for $X$ along $\xi$ at $x$
  and $\Phi$ its $F$-accompanying process.  Suppose that $h$ takes paths in
  the Cameron-Martin space $\H(\R_+,T_xM)$.  Then, by choosing $F\equiv1$, we
  see that $\Phi$ itself should already be a local martingale, say
  $\Phi_s=\int_0^s\langle k_r,dZ_r\rangle$.  Thus
  \begin{equation*}
    n\mequal\int_0^\bull(dF_r)_{X_r(x)}\,X_{r\ast}\,dh_r+
    \int_0^\bull(dF_r)_{X_r(x)}\,A(X_r(x))k_r\,dr
  \end{equation*}
  which implies
  \begin{equation*}\int_0^\bull(dF_r)_{X_r(x)}\,X_{r\ast}\,dh_r+
    \int_0^\bull(dF_r)_{X_r(x)}\,A(X_r(x))k_r\,dr\equiv0,
  \end{equation*}
  i.e., $(dF_s)_{X_s(x)}\,X_{s\ast}\,\dot h_s+
  (dF_s)_{X_s(x)}\,A(X_s(x))k_s\equiv0$ for all $F$ of the above type.  Hence,
  assuming local richness of transformations $F$ of this type, we get for
  $s\geq0$,
  \begin{equation*}
    X_{s\ast}\,\dot h_s+A(X_s(x))k_s\equiv0\end{equation*}
  or 
  \begin{equation*}\dot h_s+(X_{s\ast}^{-1}A)_x\,k_s=0.
  \end{equation*}
  which means that $k$ solves the ``control problem''.
\end{remark}

Coming back to Problem \ref{MainProblem} we note that since the problem is
unaffected by changing $M$ outside of $D$, we may assume that $M$ is already
compact.  It is also enough to deal with the case $\sigma=\tau_D^\mstrut(x)$
where $D$ has smooth boundary.  

\begin{problem}[Modified Control Problem]
  \label{MainProblem1}Let
\begin{equation*}
  c_s(x)={d\over ds}C_s(x)
  =\sum_{i=1}^r(X_{s\ast}^{-1}A_i)_x^\mstrut\otimes(X_{s\ast}^{-1}A_i)_x^\mstrut.
\end{equation*}
Confining the consideration to $\R^r$-valued processes $k$ of the special form
\begin{equation}
  \label{eq:SpecialControl}
  k_s=\sum_{i=1}^r\bigl\langle(X_{s\ast}^{-1}A_i)_x^\mstrut,u_s\bigr\rangle\,e_i  
\end{equation}
for some adapted $T_xM$-valued process $u$, we observe that Problem~\ref{MainProblem} 
reduces to finding predictable $T_xM$-valued processes $u$ such that
\begin{equation}
  \label{eq:SpecialControl1}
  \left\lbrace\begin{aligned}
      \dot h_s&=c_s(x)\,u_s\\  
      h_0&=v\quad\hbox{and}\quad h_\sigma=0.
    \end{aligned}\right.
\end{equation}
\end{problem}

This Problem \ref{MainProblem1}, as well as Problem \ref{MainProblem}, have an affirmative solution
in many cases.  However, in the general situation, both problems are not
solvable under hypothesis (H1), as will be shown in the next section.

\section{Solvability of the control problem: Examples and counterexamples}\label{Sect6}
\setcounter{equation}0\noindent

We start discussing an example with solvability of the control conditions in a
non-elliptic situation.
\begin{example}\rm
  \label{ExampleR2}
  Let $M=\R^2$ and $A_0\equiv 0$, $A_1(x)=(1,0)$, $A_2(x)=(0,x_1)$.  Then
  $[A_1,A_2](x)=(0,1)$.  The solution to
  \begin{equation*}
    \delta X=A(X)\,\delta Z
  \end{equation*}
  starting from $x=(x^1,x^2)$ is given by
  \begin{equation*}
    X_t(x) = \left(x^1+Z_t^1,x^2+x^1Z_t^2+\int_0^tZ_s^1\,dZ_s^2\right).
  \end{equation*}
  Consequently
  \begin{equation*}
    \bigl(X_{s\ast}^{-1}A\bigr)(x)=\left(
      \begin{matrix} 
        1&0\\-Z_s^2&X_s^1
      \end{matrix}
    \right),
  \end{equation*}
  and the control problem at $x=0$ comes down to finding $k$ such that
  \begin{equation*}
    \dot h_s=\left(
      \begin{matrix} 
        1&0\\-Z_s^2&Z_s^1
      \end{matrix}\right)k_s,\quad h_0=v,\ h_\sigma=0,
  \end{equation*}
  and $\Bigl(\int_0^\sigma\vert k_s\vert^2\,ds\Bigr)^{1/2}\in
  L^{1+\varepsilon}$.  We may assume that $\vert v\vert=1$, and will further
  assume that $\sigma=\tau_D$ or $\sigma=\tau_D\wedge t$ where $D$ is some
  relatively compact neighbourhood of the origin in $\R^2$.  (After possibly
  shrinking $D$, we may also assume that $D$ is open with smooth boundary.)
  Note that
  \begin{equation*}
    c_s(0)=\bigl(X_{s\ast}^{-1}A\bigr)_0\,\bigl(X_{s\ast}^{-1}A\bigr)_0^\ast=
    \left(
      \begin{matrix} 
        1&\!\!\!\!-Z_s^2\\-Z_s^2&\vert Z_s\vert^2
      \end{matrix}
    \right).
  \end{equation*}
  Thus if $\lambda_{\text{\rm min}}(s)$ denotes the smallest eigenvalue of
  $c_s(0)$, then
  \begin{equation}
    \label{eq:LambdaMinBound}
    \lambda_{\text{\rm min}}(s)\geq
    \frac{(Z_s^1)^2}{1+\vert Z_s\vert^2}.
  \end{equation}
(Indeed,
let $a:=Z_s^1$, $b:=Z_s^2$, and $x:=1+|Z_s|^2=1+a^2+b^2$; then
$$\lambda_{\text{min}}(s)=\frac{x-\sqrt{x^2-4a^2}}2=\frac x2\left[1-\sqrt{1-\frac{4a^2}{x^2}}\right]\geq
\frac{a^2}x,
$$
where we used $1-\sqrt{1-x}\geq x/2$).

  We construct $h$ by solving the equation
  \begin{equation}
    \label{eq:hsDefinition}
    \dot h_s=-\varphi^{-2}(X_s,Z_s)\,c_s(0)\,{h_s\over\vert h_s\vert},\quad h_0=v,
  \end{equation}
  where $X_s=X_s(0)$ and $\varphi$ is chosen in such a way that
  \begin{equation*}
    \sigma':=\inf\{s\geq0: h_s=0\}\leq\sigma. 
  \end{equation*}
  More precisely, take $\varphi_1\in C^2(\bar D)$ with
  $\varphi_1\vert{\partial D}=0$ and $\varphi_1>0$ in $D$.  Similarly, for
  some large ball $B$ in $\R^2$ about $0$ (containing $D$), let $\varphi_2\in
  C^2(\bar B)$ with $\varphi_2\vert{\partial B}=0$ and $\varphi_2>0$ in $B$.
  Let $\varphi(x,z):=\varphi_1(x)\varphi_2(z)$.  We only deal with the case
  $\sigma=\tau_D$, the case $\sigma=\tau_D\wedge t$ is dealt with an obvious
  modification of \eqref{eq:hsDefinition}.  Now, arguing as in the elliptic
  case, one shows
  \begin{equation*}
    \int_0^\sigma\varphi^{-2}(X_s,Z_s)\,ds=\infty,\quad\text{a.s.}
  \end{equation*}
  Consequently, since $Z^1_\sigma\not=0$ with probability $1$, we may conclude
  that also
  \begin{equation*}
    \int_0^\sigma\varphi^{-2}(X_s,Z_s)\,\frac{(Z_s^1)^2}{1+\vert Z_s\vert^2}\,ds=\infty,
    \quad\text{a.s.}
  \end{equation*}
  Note that
  \begin{equation*}
    \frac{d}{ds}\vert h_s\vert=\frac{\langle\dot h_s,h_s\rangle}{\vert h_s\vert}
    =\frac{-\varphi^{-2}(X_s,Z_s)\,\langle c_s(0)h_s,h_s\rangle}{\vert h_s\vert^2},
  \end{equation*}
  and hence by means of \eqref{eq:LambdaMinBound},
  \begin{equation*}
    1-\vert h_t\vert\geq \int_0^t\varphi^{-2}(X_s,Z_s)\,\lambda_{\text{\rm min}}(s)\,ds
    \geq \int_0^t\varphi^{-2}(X_s,Z_s)\,\frac{(Z_s^1)^2}{1+\vert Z_s\vert^2}\,ds
  \end{equation*}
  which shows in particular that
  \begin{equation*}
    \sigma'\leq \inf\left\{t\geq0: 
      \int_0^t\varphi^{-2}(X_s,Z_s)\,\frac{(Z_s^1)^2}{1+\vert Z_s\vert^2}\,ds=1\right\}.
  \end{equation*}
  It remains to verify the integrability condition, i.e.,
  $\Bigl(\int_0^{\sigma'}\vert k_s\vert^2\,ds\Bigr)^{1/2}\in
  L^{1+\varepsilon}$ where
  \begin{equation*}
    k_s=-\varphi^{-2}(X_s,Z_s)\,\bigl(X_{s\ast}^{-1}A\bigr)_0^\ast\,
    {h_s\over\vert h_s\vert}.
  \end{equation*}
  But, since on the interval $[0,\sigma]$ the Brownian motion $Z$ stays in a
  compact ball $B$, and thus
  \begin{equation*}
    \left\vert\bigl(X_{s\ast}^{-1}A\bigr)_0^\ast\,{h_s\over\vert h_s\vert}\right\vert
    \leq C
  \end{equation*}
  for some constant $C$, we are left to check
  \begin{equation*}
    \Bigl(\int_0^{\sigma'}\varphi^{-4}(X_s,Z_s)\,ds\Bigr)^{1/2}\in L^{1+\varepsilon}
  \end{equation*}
  which is done as in the elliptic case.
\end{example}

Contrary to Example \ref{ExampleR2} the next example gives a negative result
showing that in general Problem \ref{MainProblem} is not always solvable.

\begin{example}\rm
  \label{Picard}
  (J.~Picard) Let $M=\R^3$ and take
  \begin{equation*}
    A_0(x)=(0,0,0),\ A_1(x)=(1,0,0),\ A_2(x)=(0,1,x^1)
  \end{equation*} 
  which obviously satisfy (H1).  Then SDE \eqref{eq:StratonovichSDE} reads as
  \begin{equation*}
    X_t(x)=x+\left(Z_t^1,\, Z_t^2,\,x^1Z_t^2+\int_0^tZ_s^1\,dZ_s^2\right).
  \end{equation*}
  In particular,
  \begin{equation*}
    (X_{t\ast}^{-1}A_1)(0)=\bigl(1,0,-Z_t^2\bigr),\quad
    (X_{t\ast}^{-1}A_2)(0)=\bigl(0,1,Z_t^1\bigr).
  \end{equation*}
  Thus \eqref{eq:ContrPr} is given by
  \begin{equation*}
    \dot h_s=\bigl(k_s^1,\,k_s^2,\,Z_s^1k_s^2-Z_s^2k_s^1\bigr)
  \end{equation*}
  where the problem is to find $h$ such that $h_0=v=(v^1,v^2,v^3)$ and
  $h_\sigma=0$. By extracting the third coordinate, we get $\int_0^\sigma
  Z_s^1k_s^2\,ds-\int_0^\sigma Z_s^2k_s^1\,ds=-v^3$.  On the other hand, an
  integration by parts yields
  \begin{equation*}
    \int_0^\sigma Z_s^2k_s^1\,ds-\int_0^\sigma Z_s^1k_s^2\,ds
    =-\int_0^\sigma h_s^1\,dZ_s^2+\int_0^\sigma h_s^2\,dZ_s^1
  \end{equation*} 
  where the condition on the integrability of $k$ implies that $-\int_0^\sigma
  h_s^1\,dZ_s^2+\int_0^\sigma h_s^2\,dZ_s^1$ is~$L^1$ with expectation equal
  to $0$.  Combining both facts, we conclude that there is no solution
  satisfying the integrability condition if $v^3\not=0$.
\end{example}

Note that if $\sigma$ is not in $L^1$, then the condition on the integrability
of $k$ does not imply any more that $\int_0^\sigma h_s^1\,dZ_s^2+\int_0^\sigma
h_s^2\,dZ_s^1$ is in $L^1$.

\begin{remark}\rm
  \label{LambdaMin}
  In Example \ref{Picard} Malliavin's covariance is explicitly given by
  \begin{align*}
    \bigl\langle C_t(0)u,u\bigr\rangle
    &=\sum_{i=1}^2\int_0^t\bigl\langle(X_{r\ast}^{-1}A_i)(0),u\bigr\rangle^2\,dr\\
    &=\int_0^t\bigl[\bigl(u^1-u^3Z_r^2\bigr)^2+\bigl(u^2+u^3Z_r^1\bigr)^2\bigr]\,dr.
  \end{align*}
  Of course, $C_t(0)-C_s(0)=\int_s^tc_r(0)\,dr$ is non-degenerate for all
  $s<t$, nevertheless $\lambdamin c_s(0)=0$ for each fixed $s$, indeed:
  \begin{equation*}
    \bigl\langle c_s(0)u,u\bigr\rangle 
    =(u^1-u^3Z_s^2\bigr)^2+\bigl(u^2+u^3Z_s^1)^2,\quad u\in T_0M.
  \end{equation*}
\end{remark}

The negative result of example \ref{Picard} depends very much on the fact that
$\sigma=\sigma_D$ is the first exit time of the diffusion from a relatively
compact neighbourhood of its starting point. The situation changes completely
if we allow arbitrarily large stopping times $\sigma$ (not necessarily exit
times from compact sets).

In the remainder of this section we give sufficient conditions for solvability 
of the control problem. We assume that diffusions with generator $L$ have
infinite lifetime, but do no longer assume that the stopping time $\sigma$ 
is of a given type. 
The question whether in this situation, given solvability of the control problem, 
the local martingales defined in Theorem \ref{LocMart} are still uniformly 
integrable martingales, needs to be checked from case to case. 

We consider the following two conditions:
\begin{condition1}
  There exists a positive constant $\alpha$ such that for any continuous (non
  necessarily adapted) process $u_t$, taking values in $\{w\in T_xM,\ \|w\|=1\}$
  and converging to $u$ almost surely,
\begin{equation}
    \label{eq:Cond}
\int_0^\infty\langle c_s(x)u_s,u_s\rangle\1_{\{\cos(
  c_s(x)u_s,u_s)>\alpha\}}\,ds=\infty\quad \hbox{a.s.}
  \end{equation}
\end{condition1}

\begin{condition2}
  There exists a positive constant $\alpha$ such that for any $u_0\in \{w\in
  T_xM,\ \|w\|=1\}$, there exists a neighbourhood $V_{u_0}$ of $u_0$ in
  $\{w\in T_xM,\ \|w\|=1\}$, such that
  \begin{equation}
    \label{eq:Strongcond}
    \int_0^\infty\inf_{u\in V_{u_0}}\left(\langle c_s(x)u,u\rangle\1_{\{\cos(
        c_s(x)u,u)>\alpha\}}\right)\,ds=\infty\quad \hbox{a.s.}
  \end{equation}
\end{condition2}

The following result is immediate:

\begin{prop}
  \label{S1}
  Condition \textup{(C2)} implies Condition \textup{(C1)}.
\end{prop}

Now we prove that the control problem is solvable under condition (C1).

\begin{prop}
  \label{S2}
  Under Condition \textup{(C1)}, the control problem is solvable. More precisely,
  considering the random dynamical system
  \begin{equation}
    \label{eq:ContrPr2}
    \left\lbrace\begin{aligned}
        \dot h_s&=(X_{s\ast}^{-1}A)_x\,k_s\\  
        h_0&=v.
      \end{aligned}\right.
  \end{equation}
  there exists a (non necessarily finite) stopping time $\sigma$ and a
  predictable $\R^r$-valued process $k\in L^2(Z)$ such that the process $h$
  given by~\eqref{eq:ContrPr2} satisfies $h_\sigma=0$, a.s.
\end{prop}
\begin{proof}

  We look for a solution of the control problem satisfying an equation of the
  type
  \begin{equation}
    \label{eq:Ctrlvarphi}
    \dot h_s=-\varphi_s\frac{1}{\|h_s\|} c_s(x)h_s
  \end{equation}
  with $c_s(x)u=\sum_{i=1}^r
  (X_{s\ast}^{-1}A_i)_x\langle(X_{s\ast}^{-1}A_i)_x,u\rangle$, and where
  $\varphi_s$ takes its values in $\{0,1\}$.

Assuming that (C1) is satisfied, we construct a sequence of 
stopping times $(T_n)_{n\ge0}$ and a continuous
process $h$ inductively as follows:
  \begin{itemize}
  \item[(i)] $T_0=0$;
  \item[(ii)] for $n\ge 0$, if $h_{T_{2n}}=0$, then
    $T_{2n+2}=T_{2n+1}=T_{2n}$.
  \item[(iii)] for $n\ge 0$, if $h_{T_{2n}}\not=0$, $h_t$ is constant on
    $[T_{2n},T_{2n+1}]$ where
$$T_{2n+1}=\inf\{t>T_{2n},\ \cos( c_t(x)h_{T_{2n}},h_{T_{2n}})>\alpha\},$$
and $h_t$ solves
\begin{equation*}
  \dot h_s=-\frac{1}{\|h_s\|}c_s(x)h_s\quad\text{on $[T_{2n+1},T_{2n+2}]$} 
\end{equation*}
where $T_{2n+2}=\inf\{t>T_{2n+1},\ \cos(
c_t(x)h_t,h_t)<\alpha/2\hbox{ or }h_t=0\}$.
\end{itemize}
Let
\begin{equation*}
  \sigma=\inf\{t>0,\ h_t=0\}\quad (=\infty\hbox{ if this set is empty}),
\end{equation*}
and for $s<\sigma$,
\begin{equation*}
  \varphi_s=\1_{\cup_n[T_{2n+1},T_{2n+2}[}(s),
\end{equation*}
\begin{equation}
  \label{eq:Defk}
  k_s=-\varphi_s\frac{1}{\|h_s\|}\sum_{i=1}^r\langle(X_{s\ast}^{-1}A_i)_x,
  h_s\rangle e_i,
\end{equation}
where $(e_1,\ldots, e_r)$ denotes the canonical basis of $\R^r$.  Then $h_t$ solves
Eq.~\eqref{eq:Ctrlvarphi}, $\dot h_s=(X_{s\ast}^{-1}A)_xk_s$, and since
$$
\|k_s\|^2=-\varphi_s\left\langle \dot
  h_s,\frac{h_s}{\|h_s\|}\right\rangle=-\frac{d}{ds}\|h_s\|,
$$
we have
\begin{equation}
  \label{eq:Bdintksqr}
  \int_0^\sigma\|k_s\|^2\,ds\le \|h_0\|.
\end{equation}
To conclude it is sufficient to prove that solutions $h_t$
satisfy $\lim_{s\to\sigma}h_s=0$.

First we remark that $h_t$ converges almost surely as $t$ tends to
$\sigma$. This is due to the fact that
$$
\|dh\|=\frac{d\|h\|}{\cos(h,dh)}=-\frac{d\|h\|}{\cos(h,c_s(x)h)}\le
-\frac{2}{\alpha}\,d\|h\|
$$
(recall $d\|h\|\le 0$); hence $h$ has a total variation bounded by
${2\|h_0\|}/{\alpha}$.

We define $u_t={h_0}/{\|h_0\|}$ on the set where $h_t$ converges to $0$ as
$t$ tends to $\sigma$, and $u_t={h_t}/{\|h_t\|}$ on the set where $h_t$
does not converge to $0$.  This provides a process which converges as $t$ tends to
$\sigma$, but which is not adapted.  On the set where $h_t$ does not converge to
$0$, we have
$$
\|h_0\|\ge-\int_0^\sigma d\|h\|\ge\int_0^\infty \langle
c_s(x)u_s,u_s\rangle\1_{\{\cos( c_s(x)u_s,u_s)>\alpha\}}\,ds,
$$
which implies, by Condition (C1), that this set has probability $0$.
\end{proof}

\begin{example}
  Consider again Example \ref{Picard}, with $M=\R^3$,
  \begin{equation*}
    A_0(x)=(0,0,0),\ A_1(x)=(1,0,0),\ A_2(x)=(0,1,x^1).
  \end{equation*}
For $u\in T_0M$, $\Vert u\Vert=1$, 
we have 
\begin{equation*}
\langle c_s(0)u,u\rangle=(u^1-u^3Z_s^2)^2+(u^2+u^3Z_s^1)^2
\end{equation*} 
and 
\begin{equation*}
\cos(c_s(0)u,u)=\frac{(u^1-u^3Z_s^2)^2+(u^2+u^3Z_s^1)^2}{\left((u^1-u^3Z_s^2)^2+(u^2+u^3Z_s^1)^2
+(-Z_s^2u^1+Z_s^1u^2+\|Z_s\|^2u^3)^2\right)^{1/2}}.
\end{equation*}
From there it is straightforward to verify that condition (C2) is realized in this case. 
With Proposition~\ref{S2} we
obtain condition (C1), and with Proposition~\ref{S1} we get 
solvability of the control problem. We stress again that now we allow $\sigma$ 
to be arbitrarily large. 
Then, contrary to the negative result of Example \ref{Picard}, 
we are able to find $h$ such that $h_0=v$, $h_\sigma=0$, $\dot
  h_s=\bigl(k_s^1,k_s^2,Z_s^1k_s^2-Z_s^2k_s^1\bigr)$, and
  $\int_0^\sigma\vert k_s\vert^2\,ds\in L^{1}$.
\end{example}

\section{Derivative Formulas in the Hypoelliptic Case}\label{Sect7}
\setcounter{equation}0\noindent In this section the results of the Sections~\ref{Sect3}
and \ref{Sect4} are extended to derive general differentiation formulas for heat
semigroups and $L$-harmonic functions in the hypo\-elliptic case.

Let again $F(\,\nbull\,,X_\tbull(x))$, $x\in M$ be a family of local
martingales where the transformation $F$ is differentiable in the second
variable with a derivative jointly continuous in both variables.  We fix $x\in
M$ and $v\in T_xM$.  Let $\sigma$ be a stopping time which is dominated by the
first exit time of $X_\tbull(x)$ from some relatively compact neighbourhood
of~$x$.  We first note that
\begin{equation}
  \label{eq:Equality}
  dF(0,\nbull\,)_xv\equiv
  \E\left[dF(\sigma,\nbull\,)_{X_\sigma(x)}\,X_{\sigma\ast}\,v\right]
\end{equation}
where $X_{\sigma\ast}$ is the derivative process at the random time $\sigma$.
Eq.~\eqref{eq:Equality} follows from the fact that the local martingale
$F(\,\nbull\,\,,X_\tbull(x))$, differentiated in the direction $v$ at~$x$, is
again a local martingale, and under the given assumptions a uniformly
integrable martingale when stopped at $\sigma$.  Our aim is to replace the
right-hand side of~\eqref{eq:Equality} by expressions not involving
derivatives of $F$.  To this end the local martingales of Section~\ref{Sect4} are
exploited.

We start with an elementary construction.  Let $D\subset M$ be a nonempty
relatively compact domain and $\varphi\in C^2(\ovl D)$ such that
$\varphi|\partial D=0$ and $\varphi>0$ on $D$.  For $x\in D$ let
\begin{equation}
  \label{eq:DefTr}
  T(s)=\int_0^s\varphi^{-2}\bigl(X_r(x)\bigr)\,dr\,,\quad s\le\tau_D(x),
\end{equation}
and
\begin{equation}
  \label{eq:Defsr}
  \sigma(r)=\inf\bigl\{s\geq0: T(s)\ge r\bigr\}\le\tau_D(x).
\end{equation}
Note that $T(r)\to \infty$ as $r\nearrow\tau_D(x)$, almost surely,
see~\cite{T-W 98}.  Fix $t_0>0$ and consider
\begin{equation}
  \label{eq:CMprocess}
  \ell_s=\frac1{t_0}\,
  \rho\left(\int_0^s\varphi^{-2}\bigl(X_r(x)\bigr)\,dr\right)v
\end{equation}
for some $\rho\in C^1(\R_+,\R)$ such that $\rho(s)=0$ for $s$ close to $0$ and
$\rho(s)=t_0$ for $s\geq t_0$.  Then $\ell_0=0$ and $\ell_s=v$ for
$s\ge\sigma(t_0)$.
 
Now for perturbations $X^\lambda$ of $X$, as in Section~\ref{Sect3}, let
\begin{equation*}
  \ell^\lambda_s=\frac1{t_0}\,
  \rho\left(\int_0^s\varphi^{-2}\bigl(X^\lambda_r(x)\bigr)\,dr\right)v
\end{equation*}
and $\sigma^\lambda(r)=\inf\bigl\{s\geq0: T^\lambda(s)\ge r\bigr\}$.  We
introduce the abbreviation $\partial_\lambda^\mstrut
=\bigl(\frac\partial{\partial\lambda_1},\dots,
\frac\partial{\partial\lambda_n}\bigr)$.  Then
$\partial_\lambda^\mstrut\bigl\vert_{\lambda=0}\ell^\lambda_s$ exists and lies
in $\cap_{p>1}L^p$, see~\cite{T-W 98}, Section~\ref{Sect4} (the arguments there before
Theorem~4.1 extend easily to general exponents $p$).  In a similar way, using
$T^\lambda\circ\sigma^\lambda=\id$, we see that
\begin{equation*}
  \partial_\lambda^\mstrut\bigl\vert_{\lambda=0}\sigma^\lambda
  =-\frac1{T'\circ\sigma}\,
  \Bigl(\partial_\lambda^\mstrut\bigl\vert_{\lambda=0}T^\lambda\Bigr)\circ\sigma.
\end{equation*}

For our applications, it is occasionally useful to modify the above
construction such that already $\ell_s=v$ for $s\ge\sigma(t_0)\wedge t$ where
$t>0$ is fixed.  This can easily be achieved by adding a term of the type
$\tan(\pi r/2t)$ to the right-hand side of \eqref{eq:DefTr} and by changing
the definition of $\ell_s$ in an obvious way.

Now let again $F(\,\nbull\,,X_\tbull(x))$ be a local martingale, 
as in Section~\ref{Sect4}, and consider the variation
\begin{equation}
  \label{eq:PerturbedMart}
  F\bigl(\,\nbull\,,X^\lambda_\tbull(x)\bigr)\,G^\lambda_\tbull
\end{equation}
of local martingales where
\begin{equation}
  G^\lambda_t=\exp\left(-\int_0^t\,\langle a_s\lambda,dZ_s\rangle
    -{1\over2}\,\int_0^t\vert a_s\lambda\vert^2\,ds\right).
\end{equation}
Then
\begin{equation*}
  n_s=dF(s,\nbull\,)_{X_s(x)}\,X_{s\ast}
  \left(\int_0^s X_{r\ast}^{-1}A\bigl(X_r(x)\bigr)\,a_r\,dr\right)
  -F\bigl(s,X_s(x)\bigr)\int_0^sa_r^\ast\,dZ_r
\end{equation*}
is a local martingale in $T_xM$.  Observe that $n$ is the derivative of
\eqref{eq:PerturbedMart} at $0$ with respect to $\lambda$, i.e.,
$n_s=\partial_\lambda^\mstrut\bigl\vert_{\lambda=0}
F\bigl(s,X^\lambda_s(x)\bigr)\,G^\lambda_s$.  In particular, taking
\begin{equation}
  \label{eq:Choicefora}
  a_s=(X_{s\ast}^{-1}A)_x^\ast,
\end{equation} 
then
\begin{equation*}
  n_s=dF(s,\nbull\,)_{X_s(x)}\,X_{s\ast}\,C_s(x)
  -F\bigl(s,X_s(x)\bigr)\int_0^s(X_{r\ast}^{-1}A)_x\,dZ_r.
\end{equation*}
This implies that also
\begin{equation*}
  N_s:=n_sh_s-\int_0^sn_r\,dh_r
\end{equation*}
is a local martingale for any $T_x^\ast M$-valued adapted process $h$ locally
of bounded variation.  We choose $h_s=C_s(x)^{-1}\ell_s$ where $\ell$ is given
by \eqref{eq:CMprocess}.  Taking expectations gives
\begin{align}
  \label{eq:PrelimFormula}
  dF(0,\nbull\,)_xv&=
  \E\left[dF(\sigma,\nbull\,)_{X_\sigma(x)}\,X_{\sigma\ast}\,v\right]\\
  \notag &=\E\left[F\bigl(\sigma,X_\sigma(x)\bigr)
    \left(\int_0^\sigma(X_{s\ast}^{-1}A)_x\,dZ_s \right)C_\sigma^{-1}(x)\,v
    +\int_0^\sigma n_s\,dh_s\right].
\end{align}
where $\sigma:=\sigma(t_0)$.  We deal separately with the term
\begin{equation}
  \label{eq:Term}
  \E\left[\int_0^\sigma n_s\,dh_s\right]
  =\E\left[\int_0^\sigma \partial_\lambda^\mstrut\bigl\vert_{\lambda=0}
    \bigl[F\bigl(s,X^\lambda_s(x)\bigr)\,G^\lambda_s\bigr]\,
    d\bigl(C_s(x)^{-1}\ell_s\bigr)\right].
\end{equation}

To avoid integrability problems, it may be necessary, as in proof of Theorem
\ref{ResultCrude}, to go through the calculation first with
\eqref{eq:Choicefora} replaced by
\begin{equation*}
  a_s^k=(X_{s\ast}^{-1}A)_x^\ast\,1_{\{s\leq\tau_k\}},
\end{equation*}
where $(\tau_k)$ is an appropriate increasing sequence of stopping times such
that $\tau_k\nearrow\sigma$, and to take the limit as $k\to\infty$ in the
final formula. Note that, without loss of generality, $\sigma$ may be assumed
to be bounded.  We shall omit this technical modification here.

We return to the term \eqref{eq:Term}. Observe that
\begin{align*}
  \E&\left[\int_0^{\sigma^\lambda}
    F\bigl(s,X^\lambda_s(x)\bigr)\,G^\lambda_s\,
    d\bigl(C^\lambda_s(x)^{-1}\ell^\lambda_s\bigr)\right]\\
  &\quad\equiv \int_0^\infty\E\left[
    1_{\{s\leq\sigma^\lambda\}}\,F\bigl(s,X^\lambda_s(x)\bigr)\,G^\lambda_s\,
    \frac d{ds}\bigl(C^\lambda_s(x)^{-1}\ell^\lambda_s\bigr)\right]ds
\end{align*}
is independent of $\lambda$.  Thus differentiating with respect to $\lambda$
at $\lambda=0$ gives
\begin{align*}
  &\E\left[\int_0^\sigma n_s\,dh_s\right]\\
  &\quad= -\E\left[\int_0^\sigma F_s\,
    d\Bigl[\partial_\lambda^\mstrut\bigl\vert_{\lambda=0}
    \bigl(C^\lambda_s(x)^{-1}\ell^\lambda_s\bigr)\Bigr] +
    \partial_\lambda^\mstrut\bigl\vert_{\lambda=0}\int_0^{\sigma^\lambda}
    F_s\,d\bigl(C_s(x)^{-1}\ell_s\bigr)\right]\\
  &\quad= -\E\left[ F_\sigma\,
    \Bigl[\partial_\lambda^\mstrut\bigl\vert_{\lambda=0}
    \bigl(C^\lambda_s(x)^{-1}\ell^\lambda_s\bigr)\Bigr]_{s=\sigma} +
    F_\sigma\,\left(\frac d{ds}\Bigl\vert_{s=\sigma}C_s(x)^{-1}\ell_s\right)
    \left(\partial_\lambda^\mstrut\bigl\vert_{\lambda=0}{\sigma^\lambda}\right)\right]
\end{align*}
where $F_s\equiv F\bigl(s,X_s(x)\bigr)$.  Note that all terms in the last line
are nicely integrable.  Substituting this back into Eq.\
\eqref{eq:PrelimFormula}, we find a formula of the wanted type:
\begin{equation}
  \label{eq:WantedFormula}
  dF(0,\nbull\,)_xv
  =\E\left[F\bigl(\sigma,X_\sigma(x)\bigr)\,\Phi_\sigma v\right]
\end{equation}
where $\Phi_\sigma$ takes values in $T^\ast_xM$ and is $L^p$-integrable for
any $1\leq p<\infty$.  Summarizing the above discussion, we conclude with the
following two theorems.

\begin{thm}
  \label{ThmOneGeneral}
  Let $M$ be a smooth manifold and $\map fM\R$ a bounded measurable
  function. Assume that {\rm(H1)} holds.  Let $x\in M$, $v\in T_xM$,
  $t>0$. Then
  \begin{equation}
    \label{eq:dSemigroupGeneral}
    d(P_tf)_xv
    =\E\Bigl[f\bigl(X_t(x)\bigr)\,1_{\{t<\zeta(x)\}}\,\Phi_tv\Bigr]  
  \end{equation}
  for the minimal semigroup $P_tf$ defined by~{\rm\eqref{eq:SemiGroupMin}}
  where $\Phi_t$ is a $T^\ast_xM$-valued random variable which is
  $L^p$-integrable for any $1\leq p<\infty$ and local in the following sense:
  For any relatively compact neighbourhood $D$ of $x$ in $M$ there is a choice
  for $\Phi_t$ which is $\SF_\sigma$-measurable where
  $\sigma=t\wedge\tau_D^\mstrut(x)$ and $\tau_D^\mstrut(x)$ is the first exit
  time of $X$ from $D$ when starting at $x$.
\end{thm}

\begin{proof}
  Let $F(\,\nbull\,,X_\tbull(x))=(P_{t-\tbull}f)\bigl(X_\tbull(x)\bigr)$.
  Then Eq.~\eqref{eq:WantedFormula} gives
  \begin{equation*}
    d(P_tf)_xv
    =\E\left[F\bigl(\sigma,X_\sigma(x)\bigr)\,\Phi_\sigma\right]
  \end{equation*}
  Again by taking into account that $(P_{t-\sigma}f)(X_\sigma(x))
  =\E^{\SF_\sigma}\bigl[f\bigl(X_t(x)\bigr)\,1_{\{t<\zeta(x)\}}\bigr]$, we get
  the claimed formula.
\end{proof}

\begin{thm}
  \label{ThmTwoGeneral}
  Let $M$ be compact with smooth boundary $\partial M\not=\emptyset$ and $u\in
  C(M)$ be $L$-harmonic on $M{\setminus}\partial M$.  Assume that {\rm(H1)}
  holds.  Let $x\in M{\setminus}\partial M$ and $v\in T_xM$.  Denote $\tau(x)$
  the first hitting time of $X_\tbull(x)$ at $\partial M$.  Then the following
  formula holds:
  \begin{equation}
    \label{eq:dHarmonicGeneral}
    (du)_xv
    =\E\bigl[u\bigl(X_{\tau(x)}(x)\bigr)\,\Phi_{\tau(x)}v\bigr]  
  \end{equation}
  where $\Phi_{\tau(x)}$ is a $T^\ast_xM$-valued random variable which is in
  $L^p$ for any $1\leq p<\infty$ and local in the following sense: For any
  relatively compact neighbourhood $D$ of $x$ in $M$ there is a choice for
  $\Phi_{\tau(x)}$ which is already $\SF_\sigma$-measurable where
  $\sigma=\tau_D^\mstrut(x)$ is the first exit time of $X$ from $D$ when
  starting at $x$.
\end{thm}

\begin{proof}
  The proof is completely analogous to the proof of
  Theorem~\ref{ThmOneGeneral}.
\end{proof}

\begin{example}[Greek Deltas for Asian Options]\label{as_opt} 
Consider the following SDE on the real line:
\begin{equation}
  \label{Eq:Asian}
dS_t=\sigma(S_t)\,dW_t+\mu(S_t)\,dt\,,
\end{equation}
where $W_t$ is a real Brownian motion.
In Mathematical Finance one likes to calculate 
so-called \textit{Greek Deltas for Asian Options\/} 
which are expressions of the form 
$$\Delta_0=\frac\partial{\partial S_0}\E[f(S_T,A_T)],\quad T>0,$$
where $S_t$ is given as solution to \eqref{Eq:Asian} and
\begin{equation}
  \label{Eq:Asian1}
A_t=\int_0^t S_r\,dr.
\end{equation}
We may convert \eqref{Eq:Asian} to Stratonovich form 
$$dS_t=\sigma(S_t)\,\delta W_t+m(S_t)\,dt$$
and consider $X_t:=(S_t,A_t)$ as a diffusion on $\R^2$. 
Then
$$d\begin{pmatrix}X^1_t\\ X^2_t\end{pmatrix}
=\begin{pmatrix}\sigma(X^1_t)\\0\end{pmatrix}\circ dW_t
+\begin{pmatrix}m(X_t^1)\\ X_t^1\end{pmatrix}\,dt
$$
with the vector fields 
$$A_0=\begin{pmatrix}m(x_1)\\ x_1\end{pmatrix},\quad
A_1=\begin{pmatrix}\sigma(x_1)\\0\end{pmatrix}.$$
Observe that 
$$[A_1,A_0]=\begin{pmatrix}\sigma(x_1) m'(x_1)-\sigma'(x_1) m(x_1)\\ \sigma(x_1)\end{pmatrix}.$$
Thus if $\sigma>0$, then $X_t=(S_t,A_t)$ defines a hypoelliptic diffusion on $\R^2$.
\end{example}

\begin{example}[Trivial example]
In the special case $\sigma>0$ constant and $\mu=0$, i.e.,
$$\left\lbrace
\begin{aligned}
dS_t&=\sigma\,dW_t\\
dA_t&=S_t\,dt,
\end{aligned}\right.
 $$
one easily checks
$$
X_{t\ast}=
\begin{pmatrix}1&0\\ t&1\end{pmatrix}\quad\text{and}\quad
X_{t\ast}^{-1}(A_1)\otimes X_{t\ast}^{-1}(A_1)=
\sigma^2\begin{pmatrix}1&-t\\ -t&t^2\end{pmatrix},
$$
and hence 
$$C_T(x)=\sigma^2\begin{pmatrix}T&-T^2/2\\ -T^2/2&T^3/3\end{pmatrix}.$$
Consequently, the integration by parts argument of Sect.~\ref{Sect3} immediately gives
$$\frac\partial{\partial S_0}\E[f(S_T,A_T)]=
\frac6{\sigma T}\,\E\left[f(S_T,A_T)\,\left(\frac1T\int_0^TW_tdt-\frac13W_T\right)\right]
.$$
\end{example}

\begin{remark} In the more general situation of Example \ref{as_opt}, i.e.,
\begin{equation*}
dS_t=\sigma(S_t)\,dW_t+\mu(S_t)\,dt\quad\text{and}\quad
A_t=\int_0^t S_r\,dr,
\end{equation*}
Theorem \ref{ThmOneGeneral} may be applied to give a formula of the type
$$\Delta_0=\frac\partial{\partial S_0}\E[f(S_T,A_T)]
=\E[f(S_T,A_T)\,\pi_T^\mstrut],$$
where the weight $\pi_T$ is explicitely given and may be implemented numerically 
in Monte-Carlo simulations. See \cite{Friz_Cass:2010} for extensions to jump diffusions, 
and \cite{FLT:2008} for weights $\pi_T^\mstrut$ in terms of anticipating integrals.  
\end{remark}

\section{The Case of Non-Euclidean Targets}\label{Sect8}
\setcounter{equation}0\noindent The aim of this section is to adapt our
method, to some extent, to the nonlinear case of harmonic maps between
manifolds.  In addition to the manifold $M$, carrying a hypoelliptic
$L$-diffusion, we fix another manifold $N$, endowed with a torsionfree
connection $\nabla$.  In stochastic terms, a smooth map $\map uMN$ is harmonic
(with respect to~$L$) if it takes $L$-diffusions on $M$ to
$\nabla$-martingales on $N$.  Likewise, a smooth map $\map u{[0,t]\times
  M}{N}$ is said to solve the nonlinear heat equation, if
$u\bigl(t-\tbull\,,X_\tbull(x)\bigr)$ is a $\nabla$-martingale on $N$ for any
$L$-diffusion $X_\tbull(x)$ on $M$.

Henceforth, we fix a family $F(\,\nbull\,,X_\tbull(x))$, $x\in M$ of
$\nabla$-martingales on $N$ where $F$ is differentiable in the second variable
with a derivative jointly continuous in both variables.  In particular, such
transformations $F$ map hypoelliptic $L$-diffusions on $M$ into
$\nabla$-martingales on $N$ and include the following two cases:
\begin{align*}
  F(\,\nbull\,,X_\tbull(x))
  &=u\circ X_\tbull(x)\text{ for some harmonic map $\map uMN$, and}\\
  F(\,\nbull\,,X_\tbull(x))&=u\bigl(t-\tbull\,,X_\tbull(x)\bigr) \text{ where
    $u$ solves the heat equation for maps $M\to N$.}
\end{align*}

Theorem \ref{LocMart} is easily extended to this situation.  Recall that, if
$Y$ is a continuous semimartingale taking values in a manifold $N$ endowed
with a torsionfree connection~$\nabla$, then the geodesic (damped or deformed)
transport $\map{\Theta_{0,t}^\mstrut}{T^\mstrut_{\!Y_0}N}{T^\mstrut_{\!Y_t}N}$
on $N$ along $Y$ is defined by the following covariant equation along~$Y$:
\begin{equation}
  \label{eq:GeodTrans}
  \left\lbrace\begin{aligned}
      d\,(\itr0\bull\Theta_{0,\bull}^\mstrut)&=
      -\textstyle{1\over2}\,\itr0\bull
      R(\Theta_{0,\bull}^\mstrut,dY)dY\\
      \Theta_{0,0}^\mstrut&=\id
    \end{aligned}\right.
\end{equation}
where $\map{\tr0t}{T^\mstrut_{\!Y_0}N}{T^\mstrut_{\!Y_t}N}$ is parallel
translation on $N$ along $Y$ and $R$ the curvature tensor to $\nabla$,
see~\cite{A-T 97}.  Finally, recall the notion of anti-development of $Y$,
resp.\ ``deformed anti-development'' of $Y$,
\begin{equation}
  \label{eq:AntiDev}
  \SA(Y)=\int_0^{\bull}\itr0s\,\delta Y_s, \quad
  \SA_{\hbox{\sevenrm def}}(Y)=\int_0^{\bull}\Theta_{0,s}^{-1}\,\delta Y_s
\end{equation}
which by definition both take values in $T^\mstrut_{\!Y_0}N$.  Note that an
$N$-valued semimartingale is a $\nabla$-martingale if and only if $\SA(Y)$, or
equivalently $\SA_{\hbox{\sevenrm def}}(Y)$, is a local martingale.

\begin{thm}
  \label{LocMartCurvedTarget}
  Let $F(\,\nbull\,,X_\tbull(x))$, $x\in M$ be a family of
  $\nabla$-martingales on $N$, as described above.  Then, for any predictable
  $\R^r$-valued process $k$ in $L^2_{\hbox{\sevenrm loc}}(Z)$,
  \begin{equation}
    \label{eq:QuasiDerCurvedTarget}
    \Theta_{0,\bull}^{-1}\,dF(\,\nbull\,,X_\tbull(x))\,(T_xX_\tbull)
    \int_0^\bull (X_{s\ast}^{-1}A)_x^\mstrut k_s\,ds
    -\SA_{\hbox{\sevenrm def}}\bigl(F(\,\nbull\,,X_\tbull(x))\bigr)
    \int_0^\bull\langle k,dZ\rangle
  \end{equation}
  is a local martingale in $T_{F(0,x)}N$.  Here $\Theta_{0,\bull}$ denotes the
  geodesic transport on $N$ along the martingale $F(\,\nbull\,,X_\tbull(x))$.
\end{thm}

\begin{proof} Observe that by \cite{A-T 97},
  \begin{equation*}
    m_s:=\Theta_{0,s}^{-1}\,dF(s,\,\nbull\,)_{X_s(x)}\,X_{s\ast}
  \end{equation*}
  is local martingale taking values in $T_xM\otimes T_{F(0,x)}N$, and that by
  definition,
  \begin{equation*}
    \SA_{\hbox{\sevenrm def}}\bigl(F(\,\nbull\,,X_\tbull(x))\bigr)=
    \int_0^\bull\Theta_{0,s}^{-1}\,dF(s,\,\nbull\,)_{X_s(x)}
    \,A\bigl(X_s(x)\bigr)\,dZ_s.
  \end{equation*} 
  The rest of the (alternative) proof to Theorem \ref{LocMart} carries over
  with straight-forward modifications.
\end{proof}

It is straightforward to extend Theorem \ref{ThmOne} and Theorem \ref{ThmTwo}
to the nonlinear setting by means of the local martingale
\eqref{eq:QuasiDerCurvedTarget}.

\begin{thm}
  \label{ThmOneVar}
  Let $\map u{[0,t]\times M}{N}$ be a solution of the nonlinear heat equation,
  $x\in M$, $v\in T_xM$.  Let $D$ be a relatively compact open neighbourhood
  of $x$ and $\sigma=\tau_D^\mstrut(x)\wedge t$ where $\tau_D^\mstrut(x)$ is
  the first exit time of $X_\tbull(x)$ from $D$.  Suppose there exists an
  $\R^r$-valued predictable process $k$ such that
  \begin{equation*}
    \int_0^\sigma (X_{s\ast}^{-1}A)_x\,k_s\,ds\equiv v,\quad\hbox{a.s.}
  \end{equation*}
  and $\bigl(\int_0^\sigma\vert k_s\vert^2\,ds\bigr)^{1/2}\in
  L^{1+\varepsilon}$ for some $\varepsilon>0$.  Then the following formula
  holds:
  \begin{equation}
    \label{eq:dHeatEq}
    du(t,\,\nbull\,)_xv=\E\biggl[
    \SA_{\hbox{\sevenrm def}}\bigl(u(t-\nbull\,,X_\bull(a))\bigr)_{\sigma}
    \int_0^\sigma\langle k,dZ\rangle\biggr].  
  \end{equation}
\end{thm}

\begin{thm}
  \label{ThmTwoVar}
  Let $M$ be compact with smooth boundary $\partial M\not=\emptyset$.  For
  $x\in M{\setminus}\partial M$ let $\tau(x)$ be the first hitting time of
  $\partial M$ with respect to the process $X_\tbull(x)$.  Given $v\in T_xM$,
  we suppose that there exists an $\R^r$-valued predictable process $k$ such
  that
  \begin{equation*}
    \int_0^{\tau(x)}(X_{s\ast}^{-1}A)_x^\mstrut\,k_s\,ds\equiv v,\quad\hbox{a.s.}
  \end{equation*}
  and $\bigl(\int_0^{\tau(x)}\vert k_s\vert^2\,ds\bigr)^{1/2}\in
  L^{1+\varepsilon}$ for some $\varepsilon>0$. Then, for any $u\in
  C^\infty(M,N)$ which is harmonic on $M{\setminus}\partial M$, the following
  formula holds:
  \begin{equation}
    \label{eq:dHarmNonLin}
    (du)_xv
    =\E\biggl[\SA_{\hbox{\sevenrm def}}\bigl(u(X_\tbull(x))\bigr)_{\tau(x)}
    \int_0^{\tau(x)}\langle k,dZ\rangle\,\biggr].  
  \end{equation}
\end{thm}

Note that if $a$ is a predictable process taking values in
$T_xM\otimes(\R^r)^\ast$, as in Section~\ref{Sect4}, then
\begin{equation}
  \label{eq:QuasiDerCurvedTarget1}
  \Theta_{0,\bull}^{-1}\,dF(\,\nbull\,,X_\tbull(x))\,(T_xX_\tbull)
  \int_0^\bull (X_{s\ast}^{-1}A)_x^\mstrut a_s\,ds
  -\SA_{\hbox{\sevenrm def}}\bigl(F(\,\nbull\,,X_\tbull(x))\bigr)
  \int_0^\bull a_r^\ast \,dZ_r
\end{equation}
gives a local martingale in $T_xM\otimes T_{F(0,x)}N$.  In particular, setting
\begin{equation}
  \label{eq:StandChoice}
  a_s=(X_{s\ast}^{-1}A)_x^\ast\,1_{\{s\leq\tau\}},  
\end{equation}
where $\tau$ may be any predictable stopping time, we see that
\begin{equation}
  \label{eq:LocMart1}
  n_s=\Theta_{0,s}^{-1}\,dF(s,\nbull\,)_{X_s(x)}\,X_{s\ast}\,C_{s\wedge\tau}(x)
  -\SA_{\hbox{\sevenrm def}}\bigl(F(\,\nbull\,,X_\tbull(x))\bigr)_s
  \int_0^{s\wedge\tau}(X_{r\ast}^{-1}A)_x\,dZ_r
\end{equation}
is a local martingale.  Let
\begin{equation}
  \label{eq:YandDistY}
  Y=\SA_{\hbox{\sevenrm def}}\bigl(F(\,\nbull\,,X_\tbull(x))\bigr)\quad
  \text{and}\quad 
  Y^\lambda=\SA_{\hbox{\sevenrm def}}\bigl(F(\,\nbull\,,X^\lambda_\tbull(x))\bigr).
\end{equation}
for variations $X^\lambda(x)$ of $X(x)$, as in Section~\ref{Sect3}, and recall that,
again with the choice~\eqref{eq:StandChoice},
\begin{equation}
  \label{eq:Jacobi}
  J_s=\partial_\lambda^\mstrut\bigl\vert_{\lambda=0}
  F\bigl(s,X^\lambda_s(x)\bigr)=
  dF(s,\nbull\,)_{X_s(x)}^\mstrut\,X_{s\ast}\,C_{s\wedge\tau}(x).
\end{equation}
By definition, $J_\bull w$ is a vector field on $N$ along the martingale
$F(\,\nbull\,,X_\tbull(x))$ for each $w\in T_x^\ast M$.  Imitating the
strategy of Section~\ref{Sect7}, the idea is to differentiate
$Y^\lambda_\tbull\,G^\lambda_\tbull$ with respect to $\lambda$.

\begin{lemma}
  \label{VertParVariation}
  Keeping the notations as above, we have
  \begin{equation}
    \label{eq:DiffAntiDev}
    \hbox{\rm vert}\left[\partial_\lambda^\mstrut\bigl\vert_{\lambda=0}Y^\lambda\right]  
    =\Theta_{0,\bull}^{-1}J-J_0
    +\int_0^\bull\,\Theta_{0,s}^{-1}\,(\nabla\Theta_{0,s}^\mstrut)\,dY_s
  \end{equation}
  where $\nabla\Theta_{0,\bull}^\mstrut : T_{F(0,x)}N\to
  T_{F(\,\nbull\,,X_\bull(x))}N$ is defined by
  \begin{equation}
    (\nabla\Theta_{0,\bull}^\mstrut) u 
    =v_J^{-1}\bigl(\bigl(\Theta_{0,\bull}^{c\mstrut}\,h_{J_0}(u)\bigl){}^{\rm vert}\bigr).\end{equation}
  In particular, 
  $\hbox{\rm vert}\left[\partial_\lambda^\mstrut\bigl\vert_{\lambda=0}Y^\lambda\right]$
  and 
  $\Theta_{0,\bull}^{-1}J-J_0$
  differ only by a local martingale.
  Here~$\Theta_{0,\bull}^{c\mstrut}$ denotes the geodesic transport on $TN$ 
  along $J$ with respect to the complete lift $\nabla^c$ of the connection $\nabla$.
\end{lemma}

We are not going to prove Lemma~\ref{VertParVariation} here.  We just remark
that, again with the choice \eqref{eq:StandChoice} for the process $a$, we end
up with the following local martingale:
\begin{equation}
  \label{eq:VertMart}
  \begin{split}
    m&:=\hbox{vert}\left[\partial_\lambda^\mstrut\bigl\vert_{\lambda=0}
      (Y^\lambda G^\lambda)\right]\\
    &\phantom{:}=\Theta_{0,\bull}^{-1}J-J_0
    +\int_0^\bull\,\Theta_{0,s}^{-1}\nabla\Theta_{0,s}^\mstrut\,dY_s
    -Y\int_0^{\bull\wedge\tau}(X_{s\ast}^{-1}A)_x\,dZ_s.
  \end{split}
\end{equation}
Then a procedure along the lines of Section~\ref{Sect7} leads to a formula for
$dF(0,\,\nbull\,)_xv$ which is analogous to the linear case, but with an
additional term of the type
\begin{equation}
  \label{eq:Difference}
  \E\left[\Bigl(\int_0^\sigma
    \Theta_{0,s}^{-1}\nab{J_s}\Theta_{0,s}^\mstrut\,dY_s\Bigr)\,
    C^{-1}_\sigma(x)\,v\right]
\end{equation}
for some stopping time $\sigma$.  At the moment, it seems unclear whether it
is possible to avoid this extra term.

\section{Concluding Remarks}\label{Sect9}\setcounter{equation}0\noindent
1. The presented differentiation formulas are not intrinsic: they involve the
derivative flow which depends on the particular SDE and not just on the
generator.  It is possible to make the formulas more intrinsic by using the
framework of Elworthy, Le Jan, Li \cite{E-L-L 97}, \cite{E-LJ-L 98} on
geometry of SDEs (e.g., filtering out redundant noise and working with
connections induced by the SDE).

2. In this paper we exploited perturbations of the driving Brownian motion and
a change of measure as method for constituting variational formulas.  There
are of course other ways of performing perturbations leading to local
martingales which are related to integration by parts formulas.  For instance,
one observes that the local martingale property of $F(\,\nbull\,,X_\tbull(x))$
is preserved under
\begin{itemize}
\item[(i)] a change of measure via Girsanov's theorem,
\item[(ii)] a change of time,
\item[(iii)] rotations of the BM $Z$.
\end{itemize}
In particular, (iii) seems to be promising in the hypoelliptic context since
it leads to contributions in the direction of the bracket $[A_i,A_j]$.  So far
however, it is unclear to us how to relate such variations to regularity
results under hypoellipticity conditions.

%
%

\end{document}